\documentclass[11pt]{article}
\usepackage{amsthm}
\newtheorem{tm}{Theorem}[section]
\newtheorem{lm}[tm]{Lemma}

\newtheorem{rmk}[tm]{Remark}
\newtheorem{cor}[tm]{Corollary}

\newtheorem{??}[tm]{Question}
\newtheorem{defi}[tm]{Definition}

\newtheorem{conj}[tm]{Conjecture}

\usepackage{graphicx}
\usepackage{amssymb}
\usepackage{import}
\usepackage{xifthen}
\usepackage{pdfpages}
\usepackage{transparent}
\usepackage{color}

\font\tenmsb=msbm10
\font\sevenmsb=msbm7
\font\fivemsb=msbm5

\newfam\msbfam
\textfont\msbfam=\tenmsb
\scriptfont\msbfam=\sevenmsb
\scriptscriptfont\msbfam=\fivemsb
\def\Bbb#1{{\fam\msbfam #1}}

\font\teneufm=eufm10
\font\seveneufm=eufm7
\font\fiveeufm=eufm5
\newfam\eufmfam
\textfont\eufmfam=\teneufm
\scriptfont\eufmfam=\seveneufm
\scriptscriptfont\eufmfam=\fiveeufm

\newcommand\n{\noindent}

\newcommand\comp{{\mathbb{C}}}
\newcommand\real{{\Bbb R}}

\newcommand\ball{{\mathbb{B}}}
\newcommand\sph{{\mathbb{S}}}

\title{Convex Hull Volumes in Hyperbolic 3-Space}
\author{Cameron MacMahon}
\date{April 2026}

\begin{document}

\maketitle

\begin{abstract}
In this paper we provide a geometric condition satisfied by certain closed subsets of the Riemann sphere which implies that their hyperbolic convex hulls in $\mathbb{H}^3$ have infinite volume.  As a corollary, we characterize continua in the Riemann sphere whose hyperbolic convex hulls have infinite volume, answering a question of Danny Calegari. Furthermore, we give a geometric characterization of planar self-similar sets whose hyperbolic convex hulls have infinite volume.

\end{abstract}

\section{Introduction}

It was claimed in \cite{bis} that the hyperbolic convex hull in $\mathbb{H}^3$ of a Jordan curve $\Gamma$ in the Riemann sphere $\hat{\comp}$ has infinite volume except when $\Gamma$ is a circle. More generally, Danny Calegari asked whether or not it is possible to characterize those continua (compact, connected sets) in $\hat{\comp}$ whose hyperbolic convex hulls have infinite volume. In this work we provide such a characterization:

\begin{tm}
    The volume of the hyperbolic convex hull of a continuum $K$ in the Riemann sphere is either zero or infinity, and is zero if and only if $K$ lies on a circle.
\end{tm}

This follows from Theorem 3.4 below, which gives a condition on a closed subset $E \subset \hat{\comp}$ that, when satisfied, implies that the hyperbolic convex hull of $E$ has infinite volume. This condition will be satisfied if the one dimensional Hausdorff measure of $E$ is positive and $E$ does not lie on a circle. Investigating whether or not this condition is necessary, we discuss hyperbolic convex hull volumes of planar self-similar sets. We also characterize planar self-similar sets whose hyperbolic convex hulls have infinite volume:

\begin{tm}
    The volume of the hyperbolic convex hull of a self-similar set $K \subset \hat{\comp}$ is either zero or infinity, and is zero if and only if $K$ lies on a circle.
\end{tm}

 It is the author's belief that the methods in this paper can be tweaked to decide the question of finite or infinite hyperbolic convex hull volume for other classes of sets in the plane and/or particular sets. In the case of continua motivation for Theorem 1.1 comes from hyperbolic geometry, in particular the discussion following Theorem 6.24 in \cite{cal}. The proof of Theorem 6.24 (whose statement is beyond the scope of this work) would be greatly simplified if one could prove the following $``$No Wandering Continua" conjecture (Conjecture 6.29 in \cite{cal}):

\begin{conj}
    Let $G$ be a co-compact Kleinian group. Suppose $K \subset \hat{\comp}$ is a continuum which is disjoint from all of its $G$-translates. Then $K$ consists of a single point. 
\end{conj} 

As stated in \cite{cal}, it is possible to prove Conjecture 1.3 in the case when $K$ contains an arc of a round circle by a theorem of Zeghib \cite{zeg}. Calegari \cite{cal2} suggested proving Theorem 1.1 so that it suffices to go about proving Conjecture 1.3  with the additional assumption that the hyperbolic convex hull of $K$ has infinite volume. 

\bigskip

A brief remark on notation. When $A$ and $B$ are quantities that both depend on some parameter, we use the symbol $A \gtrsim B$ to denote that there exists a constant $C$ independent of the parameter so that $CA \geq B$. The notation $A \lesssim B$ has a similar meaning. We take $A \asymp B$ to mean that there is a constant $c$ independent of the parameter so that $c^{-1}A \leq B \leq cA$. 

\bigskip

\noindent $\bf{Acknowledgements:}$ I would like to thank my advisor Christopher Bishop for his support and guidance during the course of this project, as well as for reading multiple drafts of the paper. I also thank Danny Calegari for kindly explaining the relation of this paper's contents to Conjecture 1.3.

\section{Geometry of Cuspidal Hulls}

\n

We begin with some notation (see Figure $1$). Denote by $\ball$  hyperbolic 3-space in the Poincar\'{e} ball model, and let $\sph$ denote its 2-sphere at infinity. Let  $D_\delta = B(0, \delta)$ be the ball of hyperbolic radius $\delta$ at $0$, and let $x \in \sph$. Let $\Sigma_\sigma$ stand for the Euclidean sphere of radius $(1-\sigma)$ centered at zero, and let $K_\sigma$ denote the closed Euclidean ball centered at $0$ of radius $(1-\sigma)$. Finally, let $C_x$ be the hyperbolic convex hull of $D$ and $x$, and let $\gamma_x$ stand for the unique geodesic ray issuing from $0$ and landing at $x$. Note that $C_x$ is a hyperbolic $  ``$cusp" neighborhood of $\gamma_x$. 

\bigskip

\begin{figure}[ht]
    \centering
    \def\svgwidth{.6\textwidth}
\begingroup%
  \makeatletter%
  \providecommand\color[2][]{%
    \errmessage{(Inkscape) Color is used for the text in Inkscape, but the package 'color.sty' is not loaded}%
    \renewcommand\color[2][]{}%
  }%
  \providecommand\transparent[1]{%
    \errmessage{(Inkscape) Transparency is used (non-zero) for the text in Inkscape, but the package 'transparent.sty' is not loaded}%
    \renewcommand\transparent[1]{}%
  }%
  \providecommand\rotatebox[2]{#2}%
  \newcommand*\fsize{\dimexpr\f@size pt\relax}%
  \newcommand*\lineheight[1]{\fontsize{\fsize}{#1\fsize}\selectfont}%
  \ifx\svgwidth\undefined%
    \setlength{\unitlength}{841.88976378bp}%
    \ifx\svgscale\undefined%
      \relax%
    \else%
      \setlength{\unitlength}{\unitlength * \real{\svgscale}}%
    \fi%
  \else%
    \setlength{\unitlength}{\svgwidth}%
  \fi%
  \global\let\svgwidth\undefined%
  \global\let\svgscale\undefined%
  \makeatother%
  \begin{picture}(1,0.70707071)%
    \lineheight{1}%
    \setlength\tabcolsep{0pt}%
    \put(0,0){\includegraphics[width=\unitlength,page=1]{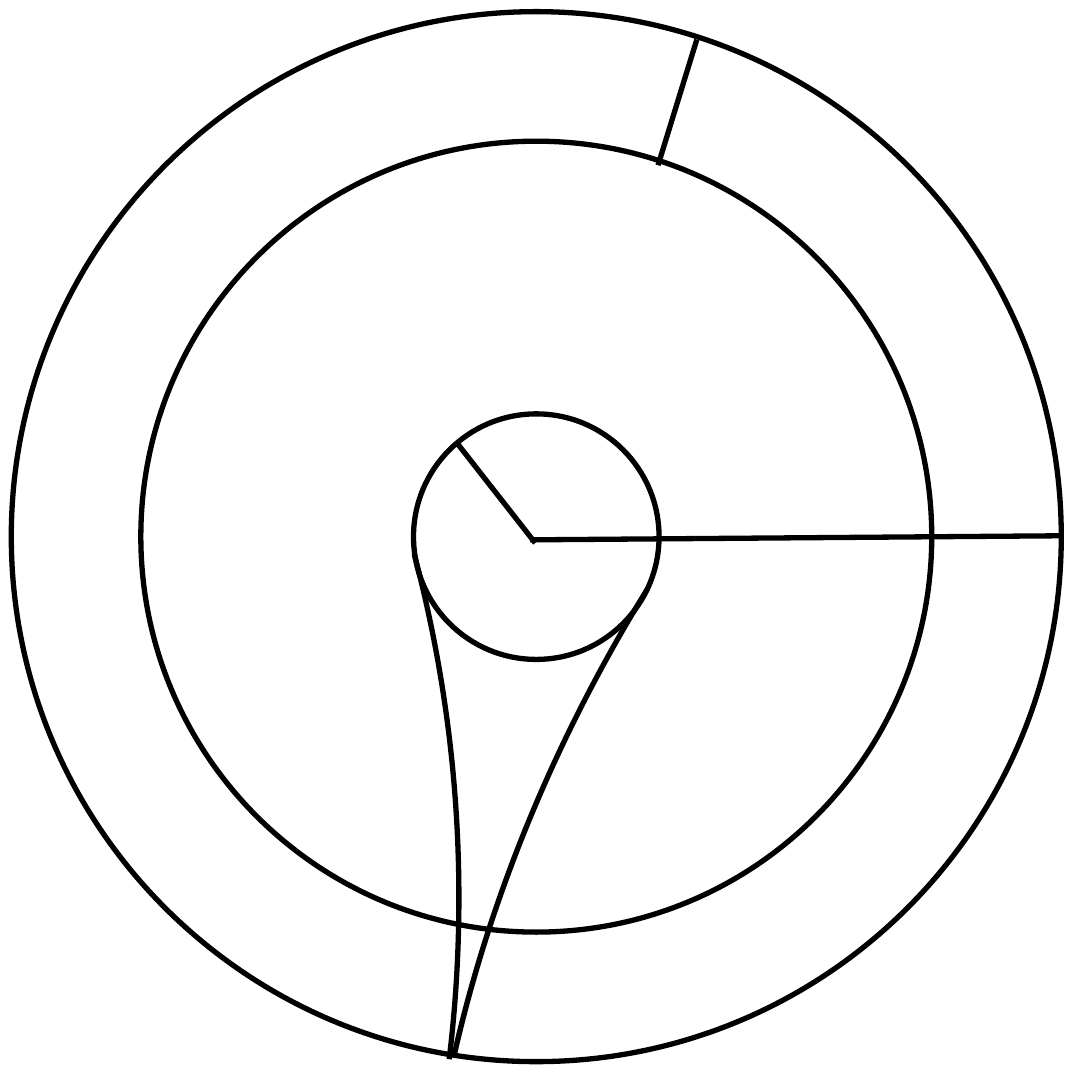}}%
    \put(0.52166316,0.43282729){\color[rgb]{0,0,0}\makebox(0,0)[lt]{\lineheight{1.25}\smash{\begin{tabular}[t]{l}$D_\delta$\end{tabular}}}}%
    \put(0.29169523,0.54141514){\color[rgb]{0,0,0}\makebox(0,0)[lt]{\lineheight{1.25}\smash{\begin{tabular}[t]{l}$\Sigma_{\sigma}$\end{tabular}}}}%
    \put(0.53592854,0.61425992){\color[rgb]{0,0,0}\makebox(0,0)[lt]{\lineheight{1.25}\smash{\begin{tabular}[t]{l}$\sigma$\end{tabular}}}}%
    \put(0.50357194,0.19678068){\color[rgb]{0,0,0}\makebox(0,0)[lt]{\lineheight{1.25}\smash{\begin{tabular}[t]{l}$C_x$\end{tabular}}}}%
    \put(0.41607371,0.04642063){\color[rgb]{0,0,0}\makebox(0,0)[lt]{\lineheight{1.25}\smash{\begin{tabular}[t]{l}$x$\end{tabular}}}}%
  \end{picture}%
\endgroup%

    \caption{A Cross Section of $C_x$}
    \label{fig:circles}
\end{figure}

 Let $V(\sigma)$ be the hyperbolic volume of $C_x - K_\sigma$, and note that this definition is independent of $x$ because all such regions are isometric. Furthermore, $C_x \bigcap \Sigma_\sigma$ is a spherical ball on $\Sigma_\sigma$ centered at $\gamma_x \bigcap \Sigma_\sigma$. Set the radius of this ball in the Euclidean spherical metric on $\Sigma_\sigma$ to be $R(\sigma)$; this is also independent of $x$. Let $x, y \in \sph$. Let $D(x,y,\sigma)$ denote the spherical Euclidean  distance between $\gamma_x \bigcap \Sigma_\sigma$  and $\gamma_y \bigcap \Sigma_\sigma$ on $\Sigma_\sigma$. We will prove the following:

 \begin{tm}
     As $\sigma \rightarrow 0$, the following hold for all $x,y \in \sph$ with constants depending only upon $\delta$:

    \begin{equation}
        \fbox{$V(\sigma) \gtrsim \sigma^2$}
    \end{equation}

    \begin{equation}
        \fbox{$R(\sigma) \asymp \sigma^2$}
    \end{equation}

    \begin{equation}
        \fbox{$D(x,y,\sigma) \asymp |x-y|$}
    \end{equation}

    where $|x-y|$ denotes the spherical distance between $x$ and $y$. 
    
 \end{tm}

 \bigskip

\begin{proof}
    We begin by proving $(1)$. We may assume that $x$ lies on the unit circle in $\hat{\comp}$ and is equal to $1$ there. Map the hyperbolic ball model $\ball$ to the upper half space model $\mathbb{H}^3$ by an isometry $\phi$ taking $0$ to $i$ and $1$ to $\infty$. Let the image of $C_x$ be called $\Gamma_x$. Since $x \mapsto \infty$ under this map $\Gamma_x$ is a Euclidean vertical cylindrical tube issuing from the sides of a small ball (exactly how small depending only on $\delta$) around $i$ tending upward toward $\infty$. Thus, up to a constant depending only on $\delta$, the tube has Euclidean radius $1$.

 \begin{figure}[ht]
    \centering
    \def\svgwidth{.6\textwidth}
\begingroup%
  \makeatletter%
  \providecommand\color[2][]{%
    \errmessage{(Inkscape) Color is used for the text in Inkscape, but the package 'color.sty' is not loaded}%
    \renewcommand\color[2][]{}%
  }%
  \providecommand\transparent[1]{%
    \errmessage{(Inkscape) Transparency is used (non-zero) for the text in Inkscape, but the package 'transparent.sty' is not loaded}%
    \renewcommand\transparent[1]{}%
  }%
  \providecommand\rotatebox[2]{#2}%
  \newcommand*\fsize{\dimexpr\f@size pt\relax}%
  \newcommand*\lineheight[1]{\fontsize{\fsize}{#1\fsize}\selectfont}%
  \ifx\svgwidth\undefined%
    \setlength{\unitlength}{841.88976378bp}%
    \ifx\svgscale\undefined%
      \relax%
    \else%
      \setlength{\unitlength}{\unitlength * \real{\svgscale}}%
    \fi%
  \else%
    \setlength{\unitlength}{\svgwidth}%
  \fi%
  \global\let\svgwidth\undefined%
  \global\let\svgscale\undefined%
  \makeatother%
  \begin{picture}(1,0.70707071)%
    \lineheight{1}%
    \setlength\tabcolsep{0pt}%
    \put(0,0){\includegraphics[width=\unitlength,page=1]{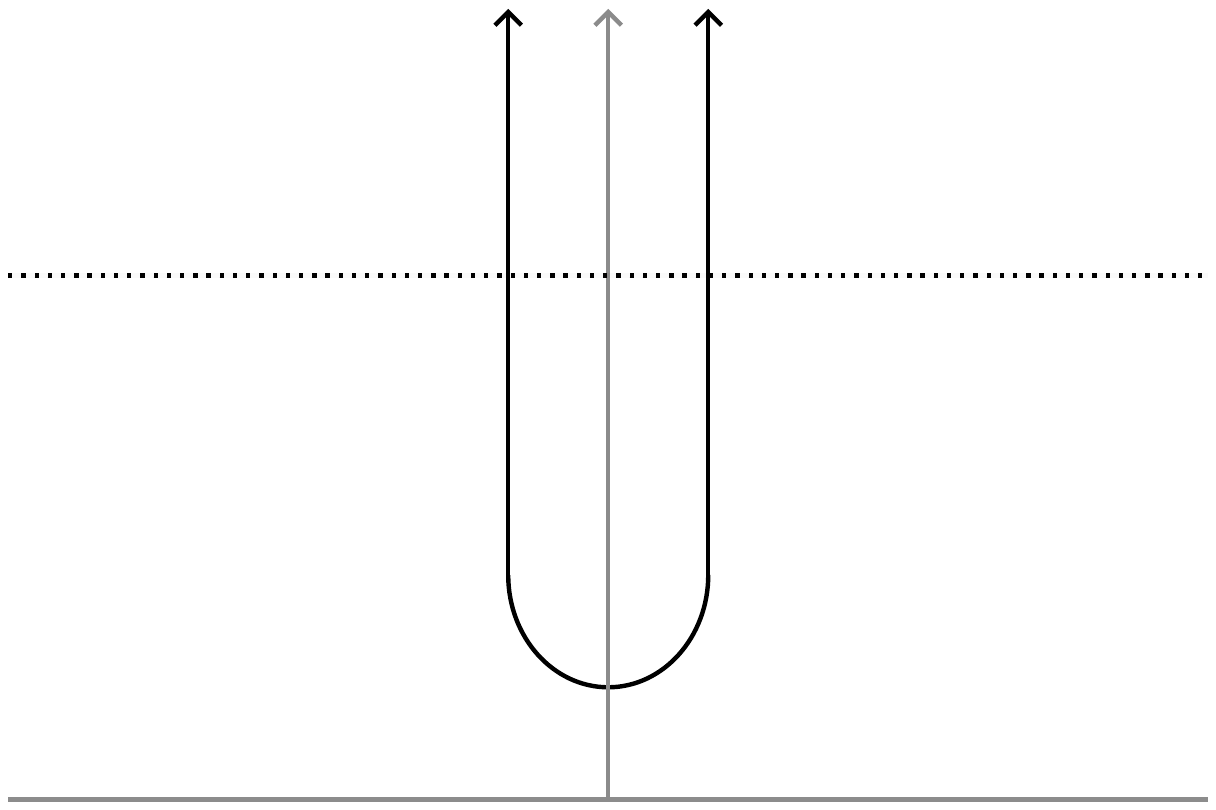}}%
    \put(0.72483196,0.40223243){\color[rgb]{0,0,0}\makebox(0,0)[lt]{\lineheight{1.25}\smash{\begin{tabular}[t]{l}$y = t$\end{tabular}}}}%
    \put(0.56581062,0.27461559){\color[rgb]{0,0,0}\makebox(0,0)[lt]{\lineheight{1.25}\smash{\begin{tabular}[t]{l}$\Gamma_x$\end{tabular}}}}%
    \put(0.15884088,0.49192174){\color[rgb]{0,0,0}\makebox(0,0)[lt]{\lineheight{1.25}\smash{\begin{tabular}[t]{l}$\mathbb{H}$\end{tabular}}}}%
    \put(0.41781839,0.55315996){\color[rgb]{0,0,0}\makebox(0,0)[lt]{\lineheight{1.25}\smash{\begin{tabular}[t]{l}$\phi(x) = \infty$\end{tabular}}}}%
  \end{picture}%
\endgroup%

    \caption{A cross section of the image of $C_x$ under $\phi$}
    \label{fig:circles}
\end{figure}

    \bigskip

 Denote the coordinate along the vertical axis in $\Bbb H^3$ by $y$. Let $t > 1$, and consider the horosphere $y = t$. We estimate the volume of $\Gamma_x \bigcap \{(x_1, x_2, y) \in \Bbb H^3 : y \geq t\}$. Since up to a constant the tube $\Gamma_x$ has width $1$, we have that a horospherical cross section $\Gamma_x \bigcap \{(x_0, x_1, y_0) \in \Bbb H^3 : y = y_0\}$, $y_0 > t$ has hyperbolic area $\frac{1}{y_0^2}$ up to a constant dependent only on $\delta$. The volume $v(t)$ of $\Gamma_x \bigcap \{(x_1, x_2, y) \in \Bbb H^3 : y \geq t\}$ is thus, with $C_\delta$ a constant depending only on $\delta$, equal to

 $$v(t) = C_\delta\int_{y=t}^\infty \frac{1}{y^2} (\frac{dy}{y}) = \frac{C_\delta}{t^2}$$

 Under the map back to $\ball$, $\Gamma_x \bigcap \{(x_1, x_2, y) \in \Bbb H^3 : y \geq t\}$ maps to $C_x \bigcap H_\sigma$ where $H_\sigma$ is the horosphere at $x=1$ tangent to $\Sigma_\sigma$, where

 $$1 - \sigma = \frac{it - i}{it + i}$$

 and so 

 $$\sigma = \frac{2}{t + 1} \quad \quad  \quad t = \frac{2-\sigma}{\sigma}$$

 Since the horosphere $H_\sigma$ lies tangent to $\Sigma_\sigma$, we have that 

 $$V(\sigma) \gtrsim v(t) = v(\frac{2-\sigma}{\sigma}) \asymp (\frac{\sigma}{2 - \sigma})^2 \rightarrow \sigma^2$$

 as $\sigma \rightarrow 0$. This completes the proof of $(1)$. 

 \bigskip

    To prove $(2)$, we recall that hyperbolic geodesics in the Poincar\'{e} model are arcs of circles orthogonal to the boundary at $x$. Assume once again that $x = 1$. Note that a hyperbolic planar cross section of $C_x$ is the region between two such circles (whose radii depend only on $\delta$) lying tangent precisely at $x = 1$. Because $\Sigma_\sigma \bigcap C_x$ is rotationally symmetric about $\gamma_x$, it suffices to prove $(2)$ for a hyperbolic planar cross section.  Within such a cross section note that the circular arc giving the intersection of $\Sigma_\sigma$ and $C_x$ is better and better approximated by segment of the line orthogonal to $\gamma_x$ between the geodesic arcs as $\sigma \rightarrow 0$. 

    \bigskip

    The Euclidean length of the segment orthogonal to $\gamma_x$ lying between the two spheres kissing at $x=1$ is $1 - \sqrt{1 - \sigma^2}$, up to a constant depending only on $\delta$.  Therefore, as $\sigma \rightarrow 0$:

    $$R(\sigma) \asymp 1 - \sqrt{1 - \sigma^2} \asymp \sigma^2$$ where we have taken the leading order term in the Taylor expansion of $\sqrt{1 - \sigma^2}$. 

    \bigskip

    Let $x,y \in \sph$. The spherical distance $D(x,y,\sigma)$ between $\gamma_x \bigcap \Sigma_\sigma$ and $\gamma_y \bigcap \Sigma_\sigma$ on $\Sigma_\sigma$ is $(1 - \sigma)|x - y|$, As $\sigma \rightarrow 0$, $D(x,y,\sigma) \rightarrow |x - y|$. This concludes the proof.     
\end{proof}

    \section{Convex Hull Volumes}

   We begin by recalling some facts about hyperbolic convex hulls. Let $E \subset \hat{\comp}$ be closed. 

    \begin{defi}
         The hyperbolic convex hull of $E$, $CH(E) \subset \ball$, is the hyperbolic convex hull of the union of all hyperbolic geodesics in $\ball$ with endpoints on $E$.
         \end{defi}

 We begin by characterizing precisely which $E$ the hyperbolic convex hull has nonempty interior. The following argument is elementary, and can be found in standard references on this topic such as Epstein and Marden \cite{epst}: 
    
    \begin{tm}
        For $E \subset \hat{\comp}$ closed, $CH(E)$ has empty interior if and only if $E$ is a subset of a circle. 
    \end{tm}

   \begin{proof}
       Suppose $E$ is a subset of a circle $C$. Then by definition $CH(E) \subset CH(C)$. However, $CH(C)$ is a dome on a sphere intersecting $\ball$ orthogonally and thus has no interior. Therefore, $CH(E)$ has no interior. 
   \bigskip

   The opposite direction follows from the fact that $CH(E)$ is a closed convex subset of $\ball$ that is not contained in any totally geodesic hyperbolic plane. For if $CH(E)$ were contained in a totally geodesic hyperbolic plane, $E$ would consist of only landing points for geodesics within this plane and those landing points would belong to the circle given by the intersection of the totally geodesic hyperbolic plane with the boundary. Thus, we may choose three non-colinear points in $CH(E)$ and form the unique totally geodesic hyperbolic plane containing the triangle they form. By the above, we may pick a fourth point in $CH(E)$ not lying on this plane. The hull of this fourth point and the triangular patch on the plane forms a hyperbolic tetrahedron contained within $CH(E)$ by convexity, and so $CH(E)$ has non-empty interior.        
   \end{proof} 

   \bigskip

    We now proceed to give a sufficient condition on a closed $E \subset \hat{\comp}$ so that $CH(E)$ has infinite hyperbolic volume. First, let us recall the notion of packing number. 

    \begin{defi}
        Let $E \subset \hat{\comp}$, and let $\epsilon > 0$. The packing number $P_\epsilon(E)$ is the maximum number of disjoint balls of radius $\epsilon$ with centers in $E$. 
    \end{defi}

    A sufficient condition on $E$ for $CH(E)$ to have infinite hyperbolic volume is the following:

    \begin{tm}
        Let $E$ be a closed subset of $\hat{\comp}$ that is not a subset of a circle and suppose there exists $\lambda \in (0,1)$ such that 

        \begin{equation}
            \sum_{j=1}^\infty \lambda^j P_{\lambda^j}(E) = \infty 
        \end{equation} Then the hyperbolic volume of $CH(E)$ is infinite. 
    \end{tm}

    \begin{proof} Let $E \subset \hat{\comp}$ be closed and not a subset of a circle. By Theorem 3.2, $CH(E)$ has non-empty interior. Furthermore, suppose that there is $\lambda \in (0,1)$ so that 
    
$$\sum_{j=1}^\infty \lambda^j P_{\lambda^j}(E) = \infty $$

    We assume $E$ is a compact subset of $\comp$. The proof admits an obvious modification in the case that $E$ is unbounded. Since $E$ is compact, let $a_{n,i}$, $1 \leq i \leq P_{\lambda^n}(E)$ label the centers of balls in a maximal packing of $E$ by balls with radii $\lambda^n$. We may assume that $0 \in \ $$
    \rm{int}$$(CH(E))$. This is because we can pick a geodesic intersecting $0$, intersecting the interior of $CH(E)$ in an open segment, and  landing at infinity. We may consider the orbit of this open segment under application of the parabolic isometry (and its inverse) given by translating along this geodesic. These open intervals cover the entire geodesic, and so in particular we may translate an interior point of $CH(E)$ to $0$ by a hyperbolic isometry fixing infinity.  Condition (4) is maintained for a potentially different $\lambda \in (0,1)$ because the hyperbolic isometries fixing infinity act on $\comp$ as affine transformations.  Consider the hyperbolic ball $B(0, \delta)$. Let $C_{n,i} = C_{a_{n,i}}$ be the convex hull of $B(0, \delta)$ and $a_{n,i}$, and let $\gamma_{n,i}$ be the unique geodesic ray joining $0$ to $a_{n,i}$.

    \bigskip

    Fix $n \geq 1$ and let $\sigma_n = \sqrt{\lambda^n}$. We consider $T_{n,i} := C_{n,i} - K_{\sigma_n}$ (the part of the tube $C_{n,i}$ lying beyond $\Sigma_{\sigma_n}$). For $n > m$, let  $T_{m,n,i} := C_{m,i} - C_{n,i}$ (the part of the tube $C_{m,i}$ lying between $\Sigma_{\sigma_m}$ and $\Sigma_{\sigma_n}$). We now aim to construct subsets of $CH(E)$ whose hyperbolic volume is bounded below (up to some constant) by 

    $$ \sum_{j=1}^n \lambda^j P_{\lambda^j}(E) \rightarrow  \infty$$ 
    
    We apply Lemma 2.1. Note that as $n \rightarrow \infty$, $\sigma_n \rightarrow 0$ because $\lambda < 1$. Therefore, the hyperbolic volume of $T_{n,i}$ is independent of $i$ and equal to $V(\sigma_n) \gtrsim \lambda^n$ as $n \rightarrow \infty$. Moreover, there are $P_{\lambda^n}(E)$ such truncated tubes, and they are all disjoint because $R(\sigma_n) \asymp \sigma_n^2 = \lambda^n$ and $D(a_{n,i}, a_{n,j}, \sigma_n) \asymp |a_{n,i} - a_{n,j}| > \lambda^n$ since $a_{n,j}$ are centers of a $\lambda^n$-packing of $E$. Now, consider also the tubes $T_{n-1,n, i}$ obtained from an optimal $\lambda^{n-1}$ packing of $E$. Suppose $n \geq N$ for $N$ some very large natural number. The volume of $T_{n-1, n,i}$ is some fixed constant (depending only on $N$) times $\sigma_{n-1}^2$ as $n \rightarrow \infty$. This is because the error in volume (given by adding the rest of the cusp beyond $\Sigma_{\sigma_n}$) shrinks to zero as $n \rightarrow \infty$. In general, for the tubes above the centers of an optimal $\lambda^{n-j}$ packing the volume of $T_{n-j, n-j+1, i}$ are also the same fixed constant times $\sigma_{n-j}^2$. Moreover, the sets $T_{n,i}$, $T_{n-j, n-j+1,i}, 1 \leq j \leq n$ are pairwise disjoint. For fixed $j$, $i$ variable $T_{n-j, n-j+1,i}$ are disjoint because of lemma $2.1$ and the fact that these tubes land at the centers of a $\lambda^{n-j}$-packing of $E$. $T_{n - j, n-j+1,i}$ and $T_{n-k, n-k+1}$, $i \neq k$ are disjoint because they lie between different $\Sigma_{\sigma_k}$ surfaces. 

    \bigskip
    
    Recall that $\sigma_j^2 = \lambda^j$. Note also that the hyperbolic volume of $K_{\sigma_N}$ is finite. Thus the hyperbolic volume of the $P_{\lambda^n}(E)$ tubes $T_{n,i}$ each of volume controlled below by  $\lambda^n$ plus the hyperbolic volume of the  $T_{n-1, n, i}$  each of volume controlled below by  $\lambda^{n-1}$ plus the volumes of the $T_{n-2, n-1, i}$ tubes etc. is at least some fixed constant depending only on $N$ and $\delta$ times

     $$\sum_{j = 1}^n \lambda^j P_{\lambda^j}(E)$$

    Therefore, by the assumption $(4)$, $CH(E)$ has subsets whose hyperbolic volumes are arbitrarily large, and the proof is concluded.
    \end{proof}

    \bigskip

    We now recall some standard notions of dimension \cite{mat}:

    \begin{defi}
       Let $K \subset \comp$ and let $\epsilon > 0$. Define the $\emph{upper and lower Minkowski Dimension}$ of $K$ to be 

    $$\overline{\rm{dim}}_M(K) = \limsup_{\epsilon \rightarrow 0} \frac{\log(P_\epsilon(K))}{\log(\frac{1}{\epsilon})}$$

    $$\underline{\rm{dim}}_M(K) = \liminf_{\epsilon \rightarrow 0} \frac{\log(P_\epsilon(K))}{\log(\frac{1}{\epsilon})}$$ respectively.  The $\emph{Hausdorff Dimension}$ of $K$ is 

    $$\mathrm{dim}_\mathcal{H} (K) = \mathrm{inf}\{s : \mathcal{H}^s(K) = 0\}$$ where the $s$-dimensional $\emph{Hausdorff Measure}$ is the $$\lim_{\delta \rightarrow \infty} \sup \{K \subset \bigcup_i E_i : \mathrm{diam}(E_i) \leq \delta \}$$ the supremum being taken over countable coverings of $K$ by sets of diameter less than or equal to $\delta$.  
       
    \end{defi}

    The Hausdorff measure of a set controls from above and below the asymptotic values of the packing numbers $P_\epsilon(K)$ as $\epsilon \rightarrow 0$. This immediately yields two corollaries. 

 \begin{cor}
     Let $E$  be a closed subset of $\comp$ that does not belong to a circle, and suppose $\mathcal{H}^1(E) > 0$. Then $CH(E)$ has infinite hyperbolic volume. 
 \end{cor}

 \begin{proof} As $j\rightarrow \infty$, we have for such an $E$ that $P_{\lambda^j}(E) \gtrsim (\frac{1}{\lambda^j})^{\dim_{\mathcal{H}}(E)}$. Theorem 3.4 may be applied.
 \end{proof}

 \begin{cor}
     Let $E$ be a continuum in the plane that does not belong to a circle. Then $CH(E)$ has infinite hyperbolic volume.
 \end{cor}

\begin{proof} The $\mathcal{H}^1$-measure of a continuum $K$ in the plane is greater than or equal to the diameter of $K$. Since $K$ does not belong to a circle, $K$ may not be a point. Therefore, $\mathcal{H}^1(K) \geq \rm{diam}$$(K) > 0$, and Corollary 3.6 applies. 
\end{proof}

\begin{rmk}
    It is tempting to believe that $\dim_\mathcal{H}(E) \geq 1$ implies that $P_{\lambda^j}(E) \gtrsim (\frac{1}{\lambda^j})^{\dim_\mathcal{H}(E)}$, and thus that Corollary $3.6$ might be true replacing Hausdorff measure with Hausdorff dimension (a weaker condition). This is false. One can construct examples where $P_{\lambda^j}(E) = \frac{\lambda^{-j}}{j^2}$. For instance, consider performing a Cantor-like construction that removes $\frac{4^j}{j^2}$ of the $4^j$ intervals subdividing the unit interval into $4^j$ intervals at the $j$'th stage. The Hausdorff dimension will still be $1$, but the series (4) will converge.
\end{rmk}

\section{Necessity and Cantor Bridge sets}

One may inquire as to whether the sufficient condition in Theorem 3.4 is also necessary in order that the hyperbolic volume of $CH(E)$ for a closed subset of $\hat{\comp}$ is infinite. This is not the case. In what follows, we construct a family of sufficiently thin Cantor sets for which the sum in Theorem 3.4 converges for all $\lambda \in (0,1)$ but whose hyperbolic convex hull volumes are infinite. 

\bigskip

Let $C_\beta$, $0 < \beta < 1$ be the standard 4-corners Cantor set in the plane, given by the product of two Cantor sets in the interval given by removing the central interval of length $\beta$. We call $CH(C_\beta) \subset \mathbb{H}^3$ a $\emph{Cantor Bridge Set}$, as the geodesics with both landing points on the totally disconnected set $C_\beta$ may be seen as $``$bridges." We begin by showing that all such sets have infinite volume.

\begin{lm}
    For $\beta \in (0,1)$, the Cantor Bridge Set $CH(C_\beta)$ has infinite hyperbolic volume.
\end{lm}

\begin{proof} The set $C_\beta$ can be obtained from the unit square $[0,1]^2$ by first removing the complement of four squares $I_1, I_2, I_3, I_4$ (ordered from left to right, top to bottom) in the corners, and then from each removing the complement of four squares $I_{1,1}$, $I_{1,2}$, $I_{1,3}$, $I_{1,4}$, $I_{2,1}$, etc. Form the hyperbolic geodesic $\mu$ joining the two inner corners $I_1 \cap I_{1,4} \cap I_{1, 4, 4} \cap ...$ and $I_4 \cap  I_{4,1} \cap I_{4,1,1} \cap ...$ and form the hyperbolic geodesic $\nu$ joining the two outer corners $I_{2} \cap I_{2, 2} \cap I_{2,2,2} \cap...$ and $I_{3} \cap I_{3,3} \cap I_{3,3,3} \cap...$. The geodesics are indicated by the dotted segments in Figure 3. It is clear that the convex hull of these two geodesics $CH(\mu, \nu) \subset CH(C_\beta)$ is an ideal hyperbolic tetrahedron. This can be seen by the fact that the convex hull of $\nu$ and a single point on $\mu$ is a hyperbolic triangle with two vertices at infinity, and so the convex hull of $\mu$ and $\nu$ is a $1$-parameter family of such with the interior vertices lying along $\mu$. We call this tetrahedron $T$. $T$ has positive hyperbolic volume which can be computed explicitly from the ideal angles at a vertex, see \cite{miln}. 

\bigskip

 \begin{figure}[ht]
    \centering
    \def\svgwidth{.6\textwidth}
\begingroup%
  \makeatletter%
  \providecommand\color[2][]{%
    \errmessage{(Inkscape) Color is used for the text in Inkscape, but the package 'color.sty' is not loaded}%
    \renewcommand\color[2][]{}%
  }%
  \providecommand\transparent[1]{%
    \errmessage{(Inkscape) Transparency is used (non-zero) for the text in Inkscape, but the package 'transparent.sty' is not loaded}%
    \renewcommand\transparent[1]{}%
  }%
  \providecommand\rotatebox[2]{#2}%
  \newcommand*\fsize{\dimexpr\f@size pt\relax}%
  \newcommand*\lineheight[1]{\fontsize{\fsize}{#1\fsize}\selectfont}%
  \ifx\svgwidth\undefined%
    \setlength{\unitlength}{841.88976378bp}%
    \ifx\svgscale\undefined%
      \relax%
    \else%
      \setlength{\unitlength}{\unitlength * \real{\svgscale}}%
    \fi%
  \else%
    \setlength{\unitlength}{\svgwidth}%
  \fi%
  \global\let\svgwidth\undefined%
  \global\let\svgscale\undefined%
  \makeatother%
  \begin{picture}(1,0.70707071)%
    \lineheight{1}%
    \setlength\tabcolsep{0pt}%
    \put(0,0){\includegraphics[width=\unitlength,page=1]{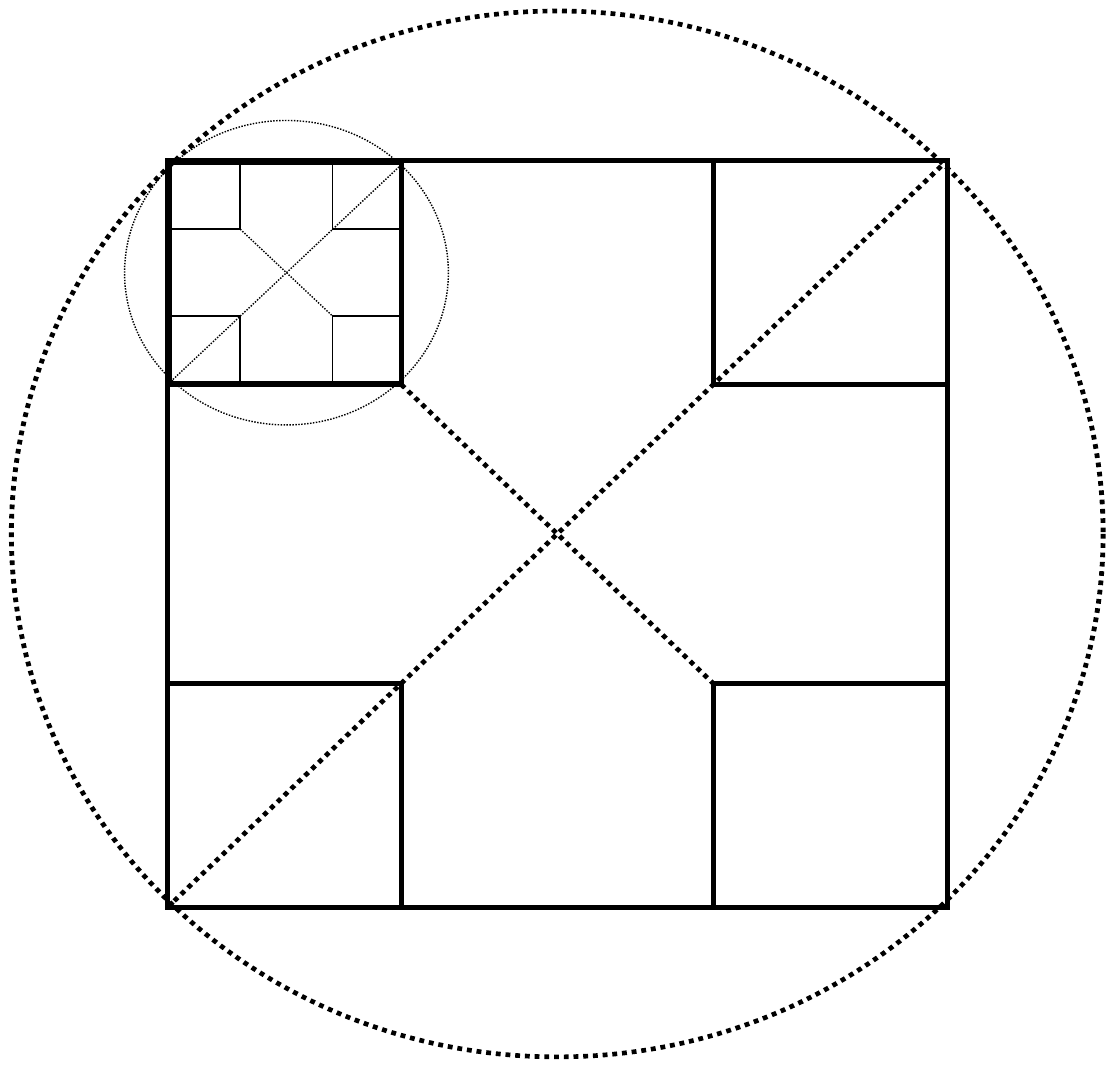}}%
    \put(0.47999193,0.07621775){\color[rgb]{0,0,0}\makebox(0,0)[lt]{\lineheight{1.25}\smash{\begin{tabular}[t]{l}$\beta$\end{tabular}}}}%
    \put(0.20992072,0.54847721){\color[rgb]{0,0,0}\makebox(0,0)[lt]{\lineheight{1.25}\smash{\begin{tabular}[t]{l}$I_1$\end{tabular}}}}%
    \put(0.74135839,0.55021956){\color[rgb]{0,0,0}\makebox(0,0)[lt]{\lineheight{1.25}\smash{\begin{tabular}[t]{l}$I_2$\end{tabular}}}}%
    \put(0.72482048,0.08921742){\color[rgb]{0,0,0}\makebox(0,0)[lt]{\lineheight{1.25}\smash{\begin{tabular}[t]{l}$I_4$\end{tabular}}}}%
    \put(0.21143016,0.09019784){\color[rgb]{0,0,0}\makebox(0,0)[lt]{\lineheight{1.25}\smash{\begin{tabular}[t]{l}$I_3$\end{tabular}}}}%
    \put(0.3731074,0.56821742){\color[rgb]{0,0,0}\makebox(0,0)[lt]{\lineheight{1.25}\smash{\begin{tabular}[t]{l}$A$\end{tabular}}}}%
  \end{picture}%
\endgroup%

    \caption{Constructing Disjoint Tetrahedra}
    \label{fig:circles}
\end{figure}

Draw the dotted circle $A$ indicated in the figure with center coinciding with the center of the upper-left corner square $I_1$ and that intersects the four corners of $I_1$, fixing the radius. Consider the hyperbolic plane bounded by $A$, a hemisphere in upper half space with boundary on the plane. Let $E$ be the open dome region in $\mathbb{H}^3$ bounded by this hemisphere and $\comp$.  By construction, $T$ does not intersect $E$. Let $S$ be the affine similarity of $\comp$ which maps $C_\beta$ onto its similar copy given by the first corner $I_1$. $S$ extends to a hyperbolic isometry $S'$ of $\mathbb{H}^3$, as $S$ is a conformal transformation of $\hat{\comp}$. Moreover by construction the image of $T$ under $S'$, call it $T'$, lies entirely in $E$. Since $S$ is a hyperbolic isometry, $T$ and $T'$ have the same hyperbolic volume. Continuing in this fashion (i.e. restricting to $I_{1,1}$ and so on) we obtain infinitely many disjoint subsets of $CH(C_\beta)$, all of the same hyperbolic volume. Therefore, $CH(C_\beta)$ has infinite hyperbolic volume.
\end{proof}

\bigskip

Given Corollary $3.6$, we know that if the $\mathcal{H}^1$-measure of $C_\beta$ is greater than zero the sufficient condition of Theorem $3.4$ is satisfied (this is shown below). We will also demonstrate that this statement is sharp. First, recall that the upper and lower Minkowski dimensions and the Hausdorff dimension of $C_\beta$ all take the same value:

$$\dim_\mathcal{H}(C_\beta) = \underline{\rm{dim}}_M(C_\beta) = \overline{\rm{dim}}_M(C_\beta) = \frac{\log(4)}{\log(2) + \log(\frac{1}{1-\beta})}$$

This formula may be obtained by computing the Minkowski dimension of the middle-$\beta$ Cantor set and using the fact that $C_\beta$ is the square of such a set \cite{bp}. This formula could also be obtained by considering $C_\beta$ as a self-similar set obeying the open set condition (see \cite{mat} \cite{bp}).

\begin{tm}
    For all $\lambda \in (0,1)$, we have 

     $$\sum_{j=1}^\infty \lambda^j P_{\lambda^j}(C_\beta) < \infty $$ if and only if $\beta > \frac{1}{2}$. 

\end{tm}

\begin{proof} Suppose $\beta \leq \frac{1}{2}$. Observe that since the upper and lower Minkowski dimensions and the Hausdorff dimension of $C_\beta$ all agree and $C_\beta$ has positive measure in that dimension, we have that as $j \rightarrow \infty$ the packing numbers obey $P_{\lambda^j}(C_\beta) \gtrsim (\frac{1}{\lambda^j})^{\dim_\mathcal{H}(C_\beta)} \asymp \lambda^{\delta j}$ for some $\delta \geq 1$. Therefore, for all $\lambda \in (0,1)$ the sum 

 $$\sum_{j=1}^\infty \lambda^j P_{\lambda^j}(C_\beta) \gtrsim \sum_{j=1}^\infty \lambda^{j(1 - \delta)} =  \infty $$
 
Since $\delta \geq 1$ and $\lambda \in (0,1)$. Now suppose $\beta > \frac{1}{2}$ and let $\lambda \in (0,1)$. Once again, since all of the notions of dimension agree and $C_\beta$ has positive Hausdorff measure in that dimension, we have that as $j\rightarrow \infty$ the packing numbers obey $P_{\lambda^j}(C_\beta) \lesssim (\frac{1}{\lambda^j})^{\dim_\mathcal{H}(C_\beta)} \asymp \lambda^{-\delta j}$ for some $\delta < 1$. Thus, we have for all $\lambda \in (0,1)$ that 

 $$\sum_{j=1}^\infty \lambda^j P_{\lambda^j}(C_\beta) \lesssim \sum_{j=1}^\infty \lambda^{j(1 - \delta)} < \infty $$

 by the geometric series formula as $\delta < 1$ and $\lambda \in (0,1)$. 

 \end{proof}

Thus, all Cantor bridge sets have infinite hyperbolic volume, but precisely half of these sets are the hyperbolic convex hull of a Cantor set which obeys the sufficient condition of Theorem 3.4. In particular, Theorem 3.4 is not a necessary condition. 

\section{General Self-Similar Sets}

The arguments given above for Cantor Bridge Sets apply in greater generality to the larger class of self similar sets. In this section, we give sufficient conditions for self similar sets to have infinite hyperbolic convex hull volumes. Recall the definition of a self-similar set \cite{mat}.

\begin{defi}
    A compact set $K \subset \comp$ is called $\emph{self-similar}$ if there exists a finite set of contracting similarities $S_i: \comp \rightarrow \comp$, $1 \leq i \leq N$ so that 

   $$K = \bigcup_{i=1}^N S_i(K) $$
    
\end{defi}

We wish to find conditions on $K$ which ensure that it is suitably $``$planar", so that arguments from Section 4 will apply. 

\begin{defi}
    Given $a,b,c,d \in \hat{\comp}$, the $\emph{cross ratio}$ of these four points, $C(a,b,c,d)$ is defined to be $M(c)$, where $M$ is the unique M\"{o}bius transformation so that $M(a) = 0$, $M(b) = 1$, and $M(d) = \infty$. 
\end{defi}

In particular, note that $C(a,b,c,d)$ is real if and only if the four points $a,b,c,d$ all lie on a circle in $\hat{\comp}$. 

\begin{lm}
    Let $a,b,c,d \in \hat{\comp}$ be such that $C(a,b,c,d) \notin \mathbb{R}$. Then $CH(\{a,b,c,d\})$ is an ideal hyperbolic tetrahedron with positive volume. 
\end{lm}

\begin{proof} Let $a,b,c,d$ be as above. $C(a,b,c,d) \notin \mathbb{R}$ implies that $a,b,c,d$ do not lie in a circle in $\hat \comp$. By Theorem $3.2$, $CH(\{a,b,c,d\})$ has nonempty interior, and therefore has positive hyperbolic volume. Moreover, by hyperbolic convex hull of four points is an ideal hyperbolic tetrahedron which in this case is non-degenerate, as it has positive volume. 
\end{proof} 
\noindent We now prove a generalization of Lemma 4.1:

\begin{tm}
    Let the compact set  $K \subset \comp$ be a self-similar set with similarities $S_i$, $1 \leq i \leq N$. Suppose there exist $a,b,c,d \in K$ with $C(a,b,c,d) \notin \mathbb{R}$. Then $CH(K)$ has infinite hyperbolic volume.
\end{tm}

 \begin{proof}We begin by proving the following weaker statement: let the compact set  $K \subset \comp$ be a self-similar set with similarities $S_i$, $1 \leq i \leq N$. Suppose there exist $a,b,c,d \in K$ with $C(a,b,c,d) \notin \mathbb{R}$. Suppose also that there exists $j \in [1,N]$ and a closed disc $A$ so that $S_j(K) \subset A$ but $a,b,c,d \notin \ $$ \rm{int}$$(A)$. Then $CH(K)$ has infinite hyperbolic volume. $CH(\{a,b,c,d\})$ is a hyperbolic tetrahedron $T$ of positive volume by Lemma 5.3. Let $U$ be the interior of the dome region in $\mathbb{H}^3$ with boundary $A \cup CH(\partial A)$. Because $a,b,c,d \notin \ $$\rm{int}$$(A)$, $T$ does not intersect $U$. Let $S_j'$ be the isometry induced on $\mathbb{H}^3$ by the (conformal) map $S_j$. Since $a,b,c,d \in K$ and $CH(S_j(K)) \subset U$ by construction, $S_j'(T)$ does not intersect $T$. Since $S_j'$ is an isometry, $S_j'(T)$ has hyperbolic volume equal to that of $T$.  Moreover, replacing $K$ with $S_j(K)$, $A$ with $S_j(A)$, and $a,b,c,d$ with $S_j(a), S_j(b), S_j(c), S_j(d)$ we may repeat the construction above to obtain infinitely many ideal hyperbolic tetrahedra of positive volume, all of which lie within $CH(K)$ as $K$ is self similar. All of these ideal hyperbolic tetrahedra are disjoint, as the iterates $S_j^{(n)}(a), S_j^{(n)}(b), S_j^{(n)}(c), S_j^{(n)}(d)$ do not belong to the interior of the $n$'th disc $S_j^{(n)}(A)$. 

\bigskip

We now demonstrate that such an $A$ always exists. Consider first a disc $D$ containing $K$. Such a disc is mapped under $S_j$, $1 \leq j \leq N$, to a disc containing $S_j(K)$. Since there is at least one similarity defining $K$, we may find a point $p \in K$ distinct from $a,b,c,d$. This point $p$ is the unique element in the countable intersection of images of $K$ under an $\emph{itinerary}$. More specifically, there exists a sequence $i_1,i_2,i_3,... \in  \{1,...,N\}$ so that $$p \in \bigcap_{j =1}^\infty S_{i_j}(K)$$

Recall, $p$ is unique because each of the $S_j(K)$ is a compact set, and as the contraction ratios are fixed, these compact sets have diameter shrinking to zero. Therefore, we may find arbitrarily small discs $D_m = S_{i_m} \circ ... \circ S_{i_1}(D)$ which contain $S_{i_m} \circ ... \circ S_{i_1}(K)$ and in particular contain $p$. Pick $M >> 1$ so that the diameter of $D_M$ is strictly less than the minimum distance among $|p-a|, |p-b|, |p-c|, |p-d|$. Then $\rm{int}$$(D_M)$ may not contain $a,b,c,d$ by definition. Without changing $K$, we may modify the collection of similarities defining $K$ to include $S_{i_m} \circ ... \circ S_{i_1}$, and we have the desired disc $A$. 
\end{proof}

\bigskip

There are therefore many self-similar sets whose hyperbolic convex hull volumes are infinite: It suffices to find a self similar set $K$ which contains four points whose cross-ratio has nonzero imaginary part. Indeed, by the Banach Contraction Mapping Principle, any finite set of contracting similarities has a unique attractor to which all compact sets tend under $K \rightarrow \bigcup_{i = 1}^N S_i(K)$. In particular, by considering one quadruple of points in $\comp$ with non-real cross ratio and one contracting similarity $S: \comp \rightarrow \comp$, we can build a $\emph{countable}$ set whose hyperbolic convex hull has infinite volume. Furthermore, the examples in Section 4 demonstrate that the sufficient condition of Theorem 3.4 may or may not be satisfied for self-similar sets. 

\bigskip

It is easy to verify for a plane set $K \subset \comp$ that if for all quadruples $(a,b,c,d) \in K \times K \times K \times K$ the cross ratios $CH(a,b,c,d) \in \mathbb{R}$, then $K$ is a subset of a circle or a line (i.e. a circle in $\hat{\comp}$). Thus we obtain a characterization of self-similar sets whose hyperbolic convex hull volumes are infinite:

\begin{cor}
    The volume of the hyperbolic convex hull of a self-similar set $K \subset \hat{\comp}$ is either zero or infinity, and is zero if and only if $K$ lies on a circle. 
\end{cor}

\end{document}